\documentclass[letterpaper,11pt,twoside,keywordsasfootnote,addressatend,noinfoline]{article}
\usepackage{fullpage}
\usepackage[english]{babel}
\usepackage{amssymb}
\usepackage{amsmath}
\usepackage{theorem}
\usepackage{epsfig}
\usepackage{subfigure}
\usepackage{mathrsfs}
\usepackage{color}
\usepackage{imsart}

\theorembodyfont{\normalfont}
\newtheorem{theorem}{Theorem}
\newtheorem{lemma}[theorem]{Lemma}
\newtheorem{definition}[theorem]{Definition}
\newtheorem{proposition}[theorem]{Proposition}

\newenvironment{proof}{\noindent{\bf Proof.}}{\hspace*{2mm}~$\square$}
\newenvironment{proofof}[1]{\noindent{\bf Proof of #1.}}{\hspace*{2mm}~$\square$}

\newcommand{\Z}{\mathbb{Z}}
\newcommand{\N}{\mathbb{N}}

\newcommand{\A}{\mathscr{A}}
\newcommand{\B}{\mathscr{B}}
\newcommand{\C}{\mathscr{C}}
\newcommand{\D}{\mathscr{D}}
\newcommand{\R}{\mathscr{R}}
\newcommand{\Lat}{\mathscr{L}}
\newcommand{\G}{\mathscr{G}}
\renewcommand{\H}{\mathscr{H}}

\newcommand{\simplex}{\mathbb{S}}
\newcommand{\uc}{\mathbf{x}}
\newcommand{\ud}{\mathbf{y}}
\newcommand{\ue}{\mathbf{z}}

\renewcommand{\b}{\beta}
\renewcommand{\c}{\mathbf{c}}
\renewcommand{\d}{\mathbf{d}}
\newcommand{\e}{\mathbf{e}}
\renewcommand{\b}{\beta}
\newcommand{\rc}{\beta_{\mathbf{c}}}
\newcommand{\rd}{\beta_{\mathbf{d}}}
\newcommand{\rcc}{\beta_{\mathbf{c}, \mathbf{c}}}
\newcommand{\rcd}{\beta_{\mathbf{c}, \mathbf{d}}}
\newcommand{\rdc}{\beta_{\mathbf{d}, \mathbf{c}}}
\newcommand{\rdd}{\beta_{\mathbf{d}, \mathbf{d}}}
\newcommand{\rci}{\beta_{\mathbf{c}, \mathbf{i}}}
\newcommand{\rdi}{\beta_{\mathbf{d}, \mathbf{i}}}

\newcommand{\is}{\mathbf{i}}
\newcommand{\js}{\mathbf{j}}
\newcommand{\ir}{\beta_{\mathbf{i}}}
\newcommand{\jr}{\beta_{\mathbf{j}}}
\newcommand{\ind}{\mathbf{1}}
\newcommand{\ep}{\epsilon}
\newcommand{\n}{\hspace*{-5pt}}

\DeclareMathOperator{\card}{card}
\DeclareMathOperator{\sign}{sign}

\DeclareMathOperator{\poisson}{Poisson \,}
\DeclareMathOperator{\exponential}{Exponential \,}

\begin{document}

\begin{frontmatter}
\title     {Local interactions promote cooperation in \\ cooperator-defector systems}
\runtitle  {Local interactions promote cooperation}
\author    {Nicolas Lanchier}
\runauthor {Nicolas Lanchier}
\address   {School of Mathematical and Statistical Sciences \\ Arizona State University \\ Tempe, AZ 85287, USA.}

\maketitle

\begin{abstract} \ \
 This paper studies a variant of the multi-type contact process as a model for the competition between cooperators and defectors on integer lattices.
 Regardless of their type, individuals die at rate one.
 Defectors give birth at a fixed rate whereas cooperators give birth at a rate that increases linearly with the number of nearby cooperators.
 In particular, it is assumed that only cooperators benefit from cooperators, which is referred to as kin-recognition in the ecological literature.
 To understand how the inclusion of space in the form of local interactions affects the dynamics, the results for the interacting particle system
 are compared with their counterpart for the non-spatial mean-field model.
 Due to some monotonicity with respect to the parameters, both the spatial and non-spatial models exhibit a unique phase transition.
 Our analysis shows however a major difference:
 In the spatial model, when cooperation is strong enough, the cooperators out-compete the defectors even when starting at arbitrarily low density.
 In contrast, regardless of the strength of cooperation, when the initial density of cooperators is too low, the defectors out-compete the cooperators
 in the non-spatial model.
 In particular, when cooperation is sufficiently strong, the cooperators can invade the defectors in their equilibrium in the spatial model but not in
 the non-spatial model, showing that space in the form of local interactions promotes cooperation.
\end{abstract}



\end{frontmatter}


\section{Introduction}
\label{sec:intro}

\indent In game theory, cooperating and defecting refer to the two possible strategies in the prisoner's dilemma, the most
 popular example of a two-person game.
 Imagine that the two accomplices of a crime, the two players, are arrested and interviewed separately by the police.
 Each player has the option to either defect by betraying his accomplice and testifying against him, or cooperate by remaining silent.
 When one player cooperates while the other one defects, the cooperator gets the smallest possible payoff called the sucker's payoff while the
 defector gets the largest possible payoff called the temptation.
 When both players cooperate they get the same payoff called the reward while when both players defect they get the same payoff called the punishment.
 The prisoner's dilemma is characterized by the conditions 
 $$ \hbox{sucker's payoff} < \hbox{punishment} < \hbox{reward} < \hbox{temptation}. $$
 Because the punishment is better than the sucker's payoff and the temptation is better than the reward, no matter the strategy of the other player,
 each player always gets a better payoff by defecting.
 Accordingly, in the context of evolutionary game theory~\cite{hofbauer_sigmund_1998, maynardsmith_price_1973} where the dynamics of a population of
 players are derived by interpreting payoff as fitness, simple models based on ordinary differential equations such as the replicator equation again
 predict that the defectors out-compete the cooperators.
 Although reasonable from a mathematical point of view, this conclusion appears to contradict the fact that cooperation is ubiquitous in nature.
 This apparent contradiction between theory and observations has been resolved through the study of more realistic models, with the notable
 example of the death-birth updating process~\cite{nowak_2006, ohtsuki_al_2006} in which each players' payoff is calculated based on a set of
 neighbors on a connected graph.
 In this model, it is possible for the cooperators to out-compete the defectors.
 The basic idea is that, due to the inclusion of space in the form of local interactions, cooperators are more likely to be located next to cooperators
 and defectors are more likely to be located next to defectors.
 Because the reward is better than the punishment, clusters of cooperators can expand linearly in space.
 This result has been proved analytically for various spatially explicit
 models~\cite{chen_2013, cox_durrett_perkins_2012, durrett_2014, evilsizor_lanchier_2016, foxall_lanchier_2017}, confirming that the inclusion of space
 in the form of local interactions promotes cooperation. \\
\indent The main objective of this paper is to again understand the effect of space in the form of local interactions in cooperator-defector systems.
 The model we study also falls in the category of interacting particle systems~\cite{liggett_1985, liggett_1999} but the evolution rules are somewhat different
 from the ones of the death-birth updating process:
 Defectors have a fixed birth rate that does not depend on the nearby configuration, indicating that defectors do not benefit from interacting
 with cooperators, while cooperators have a birth rate that is increasing with respect to the number of cooperators in their neighborhood.
 The case of interest is when the birth rate of the cooperators can be smaller or larger than the birth rate of the defectors depending on the nearby configurations. \\
\indent The interacting particle system we consider is a natural variant of the multi-type contact process inspired from the model introduced
 in~\cite[remark~3]{czuppon_pfaffelhuber_2017}.
 The latter is a variant of the biased voter model where each site of the~$d$-dimensional integer lattice is occupied by either a cooperator or a defector.
 Defectors give birth to defectors at rate one plus~$\rd$ while cooperators give birth to cooperators at rate one plus~$\rc$ times the fraction of their neighbors
 that are cooperators.
 Like in the biased voter model, the offspring replaces one of the~$2d$ parent's neighbors chosen uniformly at random.
 The parameter~$\rd$ represents the fitness advantage of the defectors over the cooperators, thus measuring the energy saved from not cooperating.
 The parameter~$\rc$ represents the benefit from cooperation, and the model assumes that only cooperators benefit from nearby cooperators which, in the
 ecological literature, is referred to as kin-recognition~\cite{czuppon_pfaffelhuber_2017}. \\
\indent Though the main objective of~\cite{czuppon_pfaffelhuber_2017} was not to study this model but two related models, their proofs easily adapt to
 show that, for the process on the integers~$\Z$ starting from a translation invariant product measure with a positive density of cooperators and defectors,
\begin{itemize}
 \item the cooperators win, i.e., the probability that any given site is occupied by a cooperator tends to one as time goes to infinity,
       if and only if~$\rc > 2 \rd$, \vspace*{2pt}
 \item the defectors win, i.e., the probability that any given site is occupied by a defector tends to one as time goes to infinity,
       if and only if~$\rc < 2 \rd$, \vspace*{2pt}
 \item the process clusters, i.e., the probability that any given bounded set contains cooperators and defectors tends to zero as time goes
       to infinity, when~$\rc = 2 \rd$.
\end{itemize}
 The authors, however, did not say much about the process in dimensions~$d > 1$.
 In contrast, our analysis holds in any dimensions. \\
\indent As mentioned above, the model we study is similar to the model in~\cite[remark~3]{czuppon_pfaffelhuber_2017} except that it is based on the multi-type
 contact process rather than the biased voter model.
 In particular, the process also includes deaths, which results in the presence of empty sites.
 More precisely, our model is a continuous-time~Markov chain whose state at time~$t$ is a spatial configuration
 $$ \xi_t : \Z^d \longrightarrow S = \{\c, \d, \e \} $$
 with~$\xi_t (x)$ denoting the state of site~$x \in \Z^d$ at time~$t$, either
 $$ \c = \hbox{occupied by a cooperator}, \quad \d = \hbox{occupied by a defector}, \quad \e = \hbox{empty}. $$
 To describe the possible transitions, for each~$\xi : \Z^d \to S$ and each~$\is \in S$, we denote by~$\xi_{x, \is}$ the configuration
 obtained from~$\xi$ by setting the state at site~$x$ equal to~$\is$ and leaving the states at all the other sites unchanged.
 The dynamics are then described by the Markov generator
\begin{equation}
\label{eq:ips}
  \begin{array}{rcl}
      Lf (\xi) & \n = \n & \displaystyle \sum_x \sum_{y \sim x} \bigg(\frac{\b}{2d} + \sum_{z \sim y} \ \frac{\rc \,\ind \{\xi (z) = \c \}}{4d^2} \bigg) \, \ind \{\xi (y) = \c, \xi (x) = \e \} \,[f (\xi_{\c, x}) - f (\xi)] \vspace*{2pt} \\
               & \n + \n & \displaystyle \sum_x \sum_{y \sim x} \bigg(\frac{\b + \rd}{2d} \bigg) \,\ind \{\xi (y) = \d, \xi (x) = \e \} \,[f (\xi_{\d, x}) - f (\xi)] \vspace*{5pt} \\
               & \n + \n & \displaystyle \sum_x \ [f (\xi_{\e, x}) - f (\xi)] \end{array}
\end{equation}
 where~$x \sim y$ means that vertex~$x$ and vertex~$y$ are nearest neighbors on the~$d$-dimensional integer lattice.
 In words, cooperators give birth to cooperators at rate~$\b$ plus~$\rc$ times the fraction of adjacent cooperators while, regardless of the nearby configuration,
 defectors give birth to defectors at the fixed rate~$\b + \rd$.
 In either case, the offspring is sent to one of the~$2d$ parent's neighbors chosen uniformly at random.
 If the target site is empty, it becomes occupied by the offspring, otherwise the offspring is removed from the system as a result of a lack of space available.
 In addition, regardless of their type, individuals die independently at rate one.
 The parameter~$\b$ is the basic birth rate of cooperators.
 Like in the model~\cite{czuppon_pfaffelhuber_2017}, the parameter~$\rd$ is the additional birth rate for defectors resulting from the energy they save from
 not cooperating while the parameter~$\rc$ measures the benefit a cooperator obtains from nearby cooperators. \\


\noindent {\bf The mean-field model}.
 Before studying the spatially explicit, stochastic model, we look at its deterministic non-spatial version called mean-field model~\cite{durrett_levin_1994}.
 That is, we assume that all sites are independent, and the system spatially homogeneous.
 This results in a system of coupled differential equations for the densities of cooperators and defectors.
 The main reason for studying this model is to compare later its behavior with the behavior of the interacting particle system in order to understand
 the effect of local interactions on the dynamics.
 Letting~$\uc$ denote the density of cooperators, and~$\ud$ the density of defectors, the mean-field model is described by
\begin{equation}
\label{eq:mean-field}
  \begin{array}{rcl}
    \uc' & \n = \n & (\b + \rc \,\uc)(1 - \uc - \ud) \,\uc - \uc \vspace*{2pt} \\
    \ud' & \n = \n & (\b + \rd)(1 - \uc - \ud) \,\ud - \ud. \end{array}
\end{equation}
 Letting also~$\ue$ be the density of empty sites, the simplex
 $$ \simplex = \{(\uc, \ud, \ue) : \uc, \ud, \ue \geq 0 \ \hbox{and} \ \uc + \ud + \ue = 1 \}, $$
 that corresponds to the set of vectors that are physically relevant, is positive invariant under the dynamics of~\eqref{eq:mean-field}.
 In particular, we are only interested in the model with initial conditions in the simplex, and the fixed points in the simplex.
 When~$\rc = \rd = 0$, the population survives in the sense that the extinction fixed point~$(0, 0, 1)$ is not the unique fixed point
 if and only if~$\b > 1$.
 Because our main objective is to study the competition between cooperators and defectors rather than whether a given type survives in the
 absence of the other type, we only study the mean-field model with~$\b > 1$.
 The following theorem describes the behavior in this case.
\begin{theorem}[mean-field model] --
\label{th:mean-field}
 Let~$\b > 1$ and
 $$ \phi (\rc) = \frac{2 \rc}{\rc + \b - \sqrt{(\rc + \b)^2 - 4 \rc}} - \b \quad \hbox{for all} \quad \rc > 0, $$
 and assume that~$\uc (0), \ud (0) > 0$. Then, we have the following.
\begin{itemize}
 \item For all~$\rd > \phi (\rc)$, the defectors win:
       $$ \begin{array}{l} \lim_{t \to \infty} \,(\uc (t), \ud (t)) =
          \displaystyle \left(0, 1 - \frac{1}{\b + \rd} \right). \end{array} $$
 \item For all~$\rd < \phi (\rc)$, the system is bistable: depending on the initial densities,
       $$ \begin{array}{l} \lim_{t \to \infty} \,(\uc (t), \ud (t)) =
          \displaystyle \left(0, 1 - \frac{1}{\b + \rd} \right) \ \hbox{or} \
          \left(\frac{\rc - \b + \sqrt{(\rc + \b)^2 - 4 \rc}}{2 \rc}, 0 \right). \end{array} $$
 \item The function~$\phi$ is nondecreasing,
       $$ \begin{array}{l} \phi (\rc) < \rc \ \hbox{for all} \ \rc > 0 \quad \hbox{and} \quad \lim_{\rc \to 0} \phi (\rc) = 0. \end{array} $$
\end{itemize}
\end{theorem}
 The theorem gives the following picture of the non-spatial deterministic mean-field model with three main properties that will also
 be studied for the stochastic model.
\begin{itemize}
 \item Regardless of the parameters~$\rc$ and~$\rd$, the cooperators cannot invade the defectors in their equilibrium.
       In contrast, the defectors can invade the cooperators, which splits the~$(\rc, \rd)$ plane into two regions:
       one where the defectors always win when starting at a positive density and one where the system is bistable and the winner
       depends on the initial densities. \vspace*{2pt}
 \item The system is monotone in that the two parameter regions are separated by a transition curve described by~$\phi$ that can be viewed
       as both the graph of a nondecreasing function of~$\rc$ and the graph of a nondecreasing function of~$\rd$.
       In words, the parameter~$\rc$ helps the cooperators whereas the parameter~$\rd$ helps the defectors. \vspace*{2pt}
 \item The defectors always win when~$\rd > \rc$ and there is a parameter region where, even if~$\rd < \rc$, the defectors
       again out-compete the cooperators.
\end{itemize}
 This is summarized in the phase diagram on the left-hand side of Figure~\ref{fig:diagram}. \\


\noindent {\bf The stochastic process}.
 Before stating our results, we give some simple observations and a couple of definitions.
 Starting from a translation invariant measure, we say that
\begin{itemize}
 \item the cooperators {\bf die out} when~$P (\xi_t (x) = \c \ \hbox{i.o.}) = 0$, \vspace*{2pt}
 \item the cooperators {\bf survive} in the long run when~$P (\xi_t (x) = \c \ \hbox{i.o.}) = 1$,
\end{itemize}
 where i.o. stands for infinitely often, i.e.,
 $$ \{\xi_t (x) = \c \ \hbox{i.o} \} = \bigcap_{t > 0} \ \bigcup_{s > t} \ \{\xi_s (x) = \c \}. $$
 Death and long-term survival of the defectors are defined similarly by replacing state~$\c$ by state~$\d$ in the events above.
 Then, we say that
\begin{itemize}
 \item the cooperators {\bf win} when the cooperators survive but the defectors die out, \vspace*{2pt}
 \item the defectors {\bf win} when the defectors survive but the cooperators die out, \vspace*{2pt}
 \item {\bf coexistence} occurs when both the cooperators and the defectors survive.
\end{itemize}
 In the absence of cooperators, defectors evolve according to the basic contact process with birth parameter~$\b + \rd$.
 In this case, it is known from~\cite{bezuidenhout_grimmett_1990} that there exists a critical value~$\b_* \in (0, \infty)$ such that,
 starting with a positive density of defectors,
 $$ \begin{array}{rcl}
    \hbox{the defectors survive} & \hbox{if and only if} & \b + \rd > \b_*. \end{array} $$
 In the absence of defectors, the evolution of the cooperators is more complicated, but the process can be coupled with the contact processes
 with parameters~$\b$ and~$\b + \rc$, respectively, to prove that, starting with a positive density of cooperators,
 $$ \begin{array}{rcl}
    \hbox{the cooperators survive} & \hbox{whenever} & \b > \b_* \vspace*{2pt} \\
    \hbox{the cooperators die out} & \hbox{whenever} & \b + \rc \leq \b_*. \end{array} $$
 Because the main objective of this paper is to study the competition between cooperators and defectors, we assume from now on that the process
 starts from a translation invariant product measure with a positive density of cooperators and a positive density of defectors.
 To prevent extinction, we further assume that~$\b > \b_*$.
 Using a coupling argument to compare processes with different birth parameters, we first prove the following monotonicity result.
\begin{theorem}[monotonicity] --
\label{th:monotonicity}
 For all~$\b$, $\xi_0$ and~$\{\is, \js \} = \{\c, \d \}$,
 $$ \begin{array}{c}
      P (\xi_t (x) = \is \ \hbox{i.o.}) \ \hbox{is nondecreasing in~$\ir$ and nonincreasing in~$\jr$}. \end{array} $$
\end{theorem}
 To understand the competition between cooperators and defectors, we fix~$\b > \b_*$ and~$\rd > 0$, and study the values of~$\rc$ for
 which the cooperators die out, survive or win.
 This and the previous monotonicity result motivate the introduction of the critical values
 $$ \begin{array}{rcl}
    \rc^- (\b, \rd) & \n = \n & \inf \,\{\rc : \hbox{the cooperators survive} \} \vspace*{4pt} \\
    \rc^+ (\b, \rd) & \n = \n & \inf \,\{\rc : \hbox{the cooperators win} \} \end{array} $$
 for the process with parameters~$\b > \b_*$ and~$\rd > 0$ starting from a translation invariant product measure with a positive density
 of cooperators and a positive density of defectors.
 Applying the monotonicity result Theorem~\ref{th:monotonicity} with~$\is = \c$ and~$\js = \d$ implies that
\begin{equation}
\label{eq:monotone-c}
  \rc^- (\b, \rd) \ \hbox{and} \ \rc^+ (\b, \rd) \ \hbox{are nondecreasing with respect to~$\rd$},
\end{equation}
 while applying the result for~$\is = \d$ and~$\js = \c$ implies that
\begin{equation}
\label{eq:monotone-d}
  \begin{array}{rcl}
                   \hbox{the defectors win} & \hbox{for all} & \rc \in [0, \rc^-) \vspace*{4pt} \\
   \hbox{cooperators and defectors coexist} & \hbox{for all} & \rc \in (\rc^-, \rc^+) \vspace*{4pt} \\
                 \hbox{the cooperators win} & \hbox{for all} & \rc \in (\rc^+, \infty). \end{array}
\end{equation}
 The main result of this paper states that, for all~$\b$ and~$\rd$, the cooperators can win provided the benefit from cooperation~$\rc$
 is sufficiently large.
 In addition, for the cooperators to survive, the benefit~$\rc$ must be larger than a value larger than~$2d \rd / (2d - 1)$.
 This can be expressed in terms of lower and upper bounds for the two critical values introduced above.
\begin{theorem}[critical values] --
\label{th:critical}
 For all~$\b > \b_*$ and~$\rd > 0$,
 $$ \rd < 2d \rd / (2d - 1) < \rc^- (\b, \rd) \leq \rc^+ (\b, \rd) < \infty. $$
\end{theorem}
 Note that~$\rc^- > 2 \rd$ in one dimension, which contrasts with the model from~\cite[remark~3]{czuppon_pfaffelhuber_2017} for which
 the critical value is equal to~$2 \rd$.
 This means that the benefit from cooperation needs to be larger for the cooperators to survive in the presence of deaths, i.e., using the multi-type
 contact process rather than the biased voter model.
 Finally, we observe that, when~$\rc = \rd = 0$, the process reduces to the multi-type contact process in which both types of particles have the same birth rate~$\b$,
 and the results from~\cite{neuhauser_1992} imply that cooperators and defectors coexist in the sense that
 $$ P (\{\xi_t (x) = \c \ \hbox{i.o.} \} \cap \{\xi_t (x) = \d \ \hbox{i.o.} \}) = 1 $$
 though there is no stationary distribution with a positive density of both types in one and two dimensions.
 This implies that
\begin{equation}
\label{eq:multitype}
  \rc^- (\b, 0) = 0 \quad \hbox{for all} \quad \b > \b_*.
\end{equation}
 Based on previous results for the biased voter model~\cite{bramson_griffeath_1980, bramson_griffeath_1981} and the multi-type contact
 process~\cite{neuhauser_1992}, we also conjecture that coexistence cannot occur in an open set of the parameter space, meaning that both critical values are equal.
 Combining this conjecture with our results, Theorems~\ref{th:monotonicity} and~\ref{th:critical} and~\eqref{eq:monotone-c}--\eqref{eq:multitype}, gives
 the phase diagram on the right-hand side of Figure~\ref{fig:diagram}. \\

\begin{figure}[t]
\centering
\scalebox{0.60}{\input{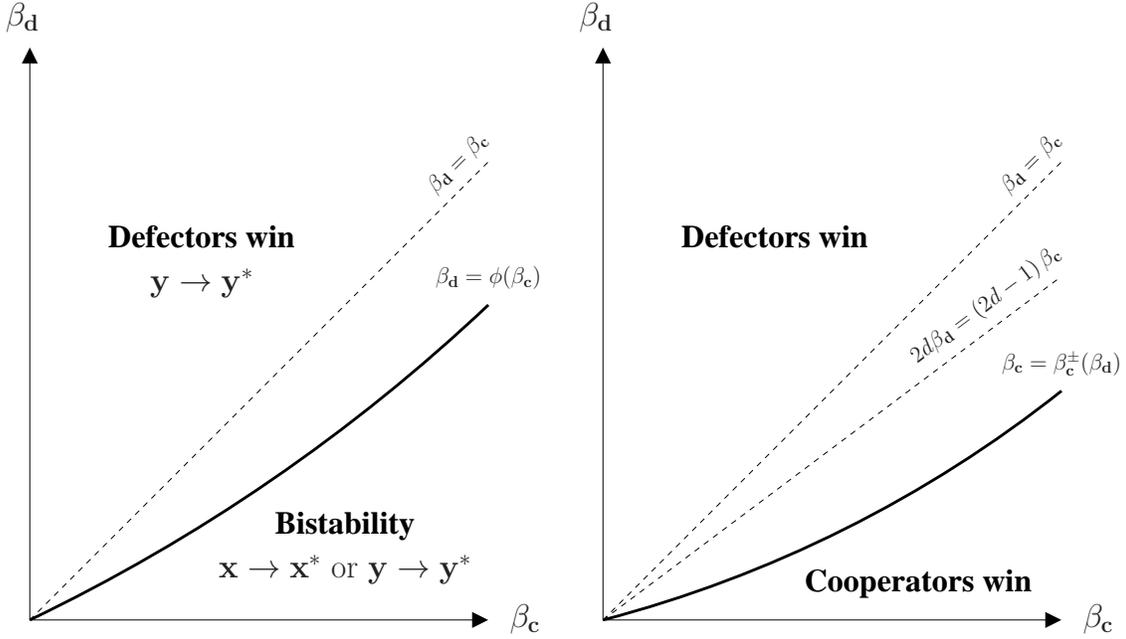}}
\caption{\upshape Phase diagram of the mean-field model on the left and phase diagram of the interacting particle system on the right in the~$(\rc, \rd)$ plane when~$\b > 1$.}
\label{fig:diagram}
\end{figure}


\noindent {\bf Effect of local interactions}.
 The spatial model~\eqref{eq:ips} and the non-spatial model~\eqref{eq:mean-field} exhibit the same monotonicity property with respect to the parameters~$\rc$
 and~$\rd$, which results in the uniqueness of a phase transition for both models.
 In addition, for both models, the defectors always win when~$\rd > \rc$ and there is a parameter region where, even if~$\rd < \rc$, the defectors
 again out-compete the cooperators.
 The two models, however, exhibit a major difference.
 In the spatial model, when cooperation is strong enough, the cooperators win starting at arbitrarily low density.
 In contrast, for all~$\rc$, the cooperators go extinct in the non-spatial model if their initial density is too low.
 In particular, for sufficiently large~$\rc$, the cooperators can invade the defectors in their equilibrium in the presence of local interactions
 but not in the absence of space, a property that is summarized in the title of this paper: Local interactions promote cooperation.


\section{The mean-field model}
\label{sec:mean-field}

\indent In this section, we study the mean-field model~\eqref{eq:mean-field} and prove Theorem~\ref{th:mean-field}.
 Recall that the mean-field model is given by the system of coupled ordinary differential equations
 $$ \begin{array}{rcccl}
     \uc' & \n = \n & f (\uc, \ud) & \n = \n & (\b + \rc \,\uc)(1 - \uc - \ud) \,\uc - \uc \vspace*{2pt} \\
     \ud' & \n = \n & g (\uc, \ud) & \n = \n & (\b + \rd)(1 - \uc - \ud) \,\ud - \ud \end{array} $$
 where~$\uc$ and~$\ud$ denote the frequencies of cooperators and defectors, respectively.
 The proof of the theorem is organized in three parts.
 First, we identify the nontrivial boundary fixed points and study their stability.
 Then, we prove that the system does not have any interior fixed point nor periodic solution in the simplex.
 Finally, we study the properties of the function~$\phi$ stated in the theorem.
 Throughout this section, we assume that~$\b > 1$. \\
 

\noindent {\bf Nontrivial boundary fixed points}.
 Setting~$\uc = 0$ and~$\ud' = 0$ gives
\begin{equation}
\label{eq:ODE-1}
  \ud = 0 \quad \hbox{or} \quad \ud = \ud^* = 1 - \frac{1}{\b + \rd}.
\end{equation}
 Then, assuming that~$\ud = \ud^*$ and~$\uc > 0$ is small, we get
\begin{equation}
\label{eq:ODE-2}
  \uc' = \b \,(1 - \ud^*) \,\uc - \uc + o (\uc) = \left(\frac{\b}{\b + \rd} - 1 \right) \uc + o (\uc) < 0
\end{equation}
 showing that, regardless of the value of~$\rc$, the defector fixed point~$(0, \ud^*, 1 - \ud^*)$ is locally stable:
 the cooperators cannot invade the defectors in their equilibrium.
 Turning to the other boundary, we set~$\ud = 0$ and~$\uc' = 0$, which gives
 $$ \uc = 0 \quad \hbox{or} \quad \uc' = (\b + \rc \,\uc)(1 - \uc) - 1 = - \rc \,\uc^2 + (\rc - \b) \,\uc + \b - 1 = 0. $$
 Thinking of~$\uc'$ as a polynomial in~$\uc$, the discriminant is
 $$ \begin{array}{rcl}
    \Delta & \n = \n & (\rc - \b)^2 + 4 (\b - 1) \,\rc = (\rc + \b)^2 - 4 \rc \vspace*{2pt} \\
           & \n > \n & (\rc + 1)^2 - 4 \rc = (\rc - 1)^2 \geq 0. \end{array} $$
 This gives two distinct roots:
\begin{equation}
\label{eq:ODE-3}
   \uc_* = \frac{\rc - \b - \sqrt{(\rc + \b)^2 - 4 \rc}}{2 \rc} \quad \hbox{and} \quad
   \uc^* = \frac{\rc - \b + \sqrt{(\rc + \b)^2 - 4 \rc}}{2 \rc}.
\end{equation}
 The fixed point~$(\uc_*, 0, 1 - \uc_*)$ is not physically relevant because
 $$ 2 \rc \,\uc_* \leq \rc - \b - \sqrt{(\rc - 1)^2} =
      \left\{\begin{array}{lcl} 2 \rc - \b - 1 < 0 & \hbox{for} & \rc \leq 1 \vspace*{2pt} \\
                                - \b + 1 < 0       & \hbox{for} & \rc \geq 1, \end{array} \right. $$
 indicating that, starting from an initial condition in the simplex, the density of cooperators~$\uc$ cannot converge to~$\uc_*$.
 In contrast, for all~$\b > 1$ and~$\rc > 0$,
 $$ \begin{array}{rcl}
      (\rc - \b)^2 - \Delta = (\rc - \b)^2 - (\rc + \b)^2 + 4 \rc & \n = \n & 4 (1 - \b) \,\rc < 0 \vspace*{2pt} \\
      (\rc + \b)^2 - \Delta = (\rc + \b)^2 - (\rc + \b)^2 + 4 \rc & \n = \n & 4 \rc > 0 \end{array} $$
 from which it respectively follows that
 $$ \begin{array}{rcl}
    \rc - \b + \sqrt{\Delta} > 0 & \hbox{and} & \uc^* > 0 \vspace*{2pt} \\
    \rc + \b - \sqrt{\Delta} > 0 & \hbox{and} & 2 \rc (\uc^* - 1) = - \rc - \b + \sqrt{\Delta} < 0. \end{array} $$
 In particular, $\uc^* \in (0, 1)$, therefore~$(\uc^*, 0, 1 - \uc^*)$ is in the simplex.
 To study the local stability of the cooperator fixed point, we observe that, when~$\uc = \uc^*$ and~$\ud > 0$ is small,
 $$ \ud' = (\b + \rd)(1 - \uc^*) \,\ud - \ud + o (\ud) $$
 which is positive if
\begin{equation}
\label{eq:ODE-4}
  \rd > \frac{1}{1 - \uc^*} - \b = \frac{2 \rc}{\rc + \b - \sqrt{(\rc + \b)^2 - 4 \rc}} - \b = \phi (\rc),
\end{equation}
 and negative if~$\rd < \phi (\rc)$. \\


\noindent {\bf Interior fixed point and periodic solution}.
 To deduce the first two parts of the theorem from the existence and stability of the two boundary fixed points above~\eqref{eq:ODE-1}--\eqref{eq:ODE-4},
 we still need to prove that the system does not have any interior fixed point nor periodic solution in the simplex.
 To prove the lack of interior fixed point, note that~$\ud \neq 0$ and~$\ud' = 0$ imply
 $$ (\b + \rd)(1 - \uc - \ud) = 1 \quad \hbox{and} \quad \ud = 1 - \uc - \frac{1}{\b + \rd}. $$
 Assuming also that~$\uc \neq 0$ and~$\uc' = 0$ gives
 $$ (\b + \rc \,\uc) \left(1 - \uc - \left(1 - \uc - \frac{1}{\b + \rd} \right) \right) = \frac{\b + \rc \,\uc}{\b + \rc} = 1 $$
 therefore~$\uc = 1$ and~$\ud < 0$.
 In particular, all the fixed points such that~$\uc \ud \neq 0$ cannot be contained in the simplex, showing that there is no interior fixed point. \\
\indent To prove the lack of periodic solution in the simplex, we use the Bendixson-Dulac theorem.
 In particular, it suffices to find a so-called Dulac function~$h (\uc, \ud)$ such that
\begin{equation}
\label{eq:ODE-5}
  \sign \left(\frac{\partial (hf)}{\partial \uc} + \frac{\partial (hg)}{\partial \ud} \right) \neq 0 \ \hbox{and is constant}
\end{equation}
 almost everywhere in the simplex.
 Taking for instance~$h (\uc, \ud) = 1 / \uc^2 \ud$, we get
 $$ \begin{array}{rcl}
    \displaystyle \frac{\partial (hf)}{\partial \uc} & \n = \n &
    \displaystyle \frac{1}{\ud} \ \frac{\partial}{\partial \uc} \left[\left(\frac{\b}{\uc} + \rc \right)(1 - \uc - \ud) - \frac{1}{\uc} \right] \vspace*{8pt} \\ & \n = \n &
    \displaystyle \frac{1}{\ud} \left[\left(- \frac{\b}{\uc^2} \right) (1 - \uc - \ud) - \left(\frac{\b}{\uc} + \rc \right) + \frac{1}{\uc^2} \right] =
    \displaystyle \frac{\b (\ud - 1) - \rc \,\uc^2 + 1}{\uc^2 \ud}. \end{array} $$
 Some basic algebra also gives
 $$ \frac{\partial (hg)}{\partial \ud} =
    \frac{1}{\uc^2} \ \frac{\partial}{\partial \ud} \left[(\b + \rd)(1 - \uc - \ud) - 1 \right] = - \frac{\b + \rd}{\uc^2}. $$
 Recalling that~$\b > 1$, we deduce
 $$ \begin{array}{rcl}
    \displaystyle \frac{\partial (hf)}{\partial \uc} + \frac{\partial (hg)}{\partial \ud} & \n = \n &
    \displaystyle \frac{\b (\ud - 1) - \rc \,\uc^2 + 1 - (\b + \rd) \,\ud}{\uc^2 \ud} \vspace*{4pt} \\ & \n = \n &
    \displaystyle - \frac{\rc \,\uc^2 + \rd \,\ud + (\b - 1)}{\uc^2 \ud} < 0 \end{array} $$
 for all~$(\uc, \ud, 1 - \uc - \ud)$ in the interior of the simplex, therefore~\eqref{eq:ODE-5} holds. \\


\noindent {\bf The transition curve}.
 To complete the proof of Theorem~\ref{th:mean-field}, the last step is to study the function~$\phi$ representing the critical curve between the two regimes:
 $$ \phi (\rc) = \frac{2 \rc}{\rc + \b - \sqrt{\Delta}} - \b = \frac{2 \rc}{\rc + \b - \sqrt{(\rc + \b)^2 - 4 \rc}} - \b. $$
 A direct calculation gives
 $$ \begin{array}{l}
     (\b \sqrt{\Delta})^2 - (\Delta - \rc \,(\rc + \b - 2))^2 \vspace*{4pt} \\ \hspace*{20pt} =
      \b^2 \,((\rc + \b)^2 - 4 \rc) - ((\rc + \b)^2 - 4 \rc - \rc \,(\rc + \b - 2))^2 \vspace*{4pt} \\ \hspace*{20pt} =
      \b^2 \,((\rc + \b)^2 - 4 \rc) - (\b \rc + \b^2 - 2 \rc)^2 = 4 (\b - 1) \,\rc^2 \end{array} $$
 from which it follows that, for all~$\b > 1$,
 $$ \frac{\partial \uc^*}{\partial \rc} = \frac{\rc \,(\rc + \b - 2) + \b \sqrt{\Delta} - \Delta}{2 \sqrt{\Delta} \,\rc^2} > 0 $$
 hence~$\uc^*$ is nondecreasing in~$\rc$.
 Recalling from~\eqref{eq:ODE-4} that
 $$ \phi (\rc) = \frac{1}{1 - \uc^*} - \b, $$
 we deduce that~$\phi$ also is nondecreasing.
 To prove that~$\phi (\rc) < \rc$, we first observe that, using basic algebra (expanding and reducing), we get
 $$ ((\rc + \b)^2 - 2 \rc)^2 > (\rc + \b)^2 \,\Delta \quad \hbox{for all} \quad \rc > 0, $$
 from which we easily deduce that
 $$ (\rc + \b)(\rc + \b - \sqrt{\Delta}) > 2 \rc \quad \hbox{and} \quad \phi (\rc) = \frac{2 \rc}{\rc + \b - \sqrt{\Delta}} - \b < \rc. $$
 Finally, note that, when~$\rc$ is small,
 $$ \begin{array}{rcl}
    \sqrt{\Delta} & \n = \n & \sqrt{(\rc + \b)^2 - 4 \rc} \\
                  & \n = \n & \displaystyle \sqrt{\b^2 + 2 (\b - 2) \,\rc + o (\rc)} = \b \,\bigg(1 + \frac{\b - 2}{\b^2} \ \rc + o (\rc)\bigg) \end{array} $$
 from which it follows that
 $$ \begin{array}{rcl}
    \lim_{\rc \to 0} \phi (\rc) & \n = \n & \lim_{\rc \to 0} \displaystyle \frac{2 \rc}{\displaystyle \rc + \b - \b \,\bigg(1 + \frac{\b - 2}{\b^2} \ \rc \bigg)} - \b \vspace*{4pt} \\
                                & \n = \n & \lim_{\rc \to 0} \displaystyle \frac{2 \b \rc}{\b \rc - (\b - 2) \,\rc} - \b = 0. \end{array} $$
 This completes the proof of Theorem~\ref{th:mean-field}.


\section{Graphical representation and monotonicity}
\label{sec:monotone}

\indent This section first shows how to construct the interacting particle system from collections of independent Poisson processes using an idea
 of Harris~\cite{harris_1972}.
 This results in a so-called graphical representation that we then use to couple cooperator-defector systems with different birth parameters and prove
 the monotonicity given in Theorem~\ref{th:monotonicity}.
 For all~$x, y, z \in \Z^d$, we let
\begin{equation}
\label{eq:harris}
    \begin{array}{rclcl}
         B (x, y) & \n = \n & \hbox{Poisson process with intensity~$\b / 2d$} & \hbox{when} & x \sim y \vspace*{4pt} \\
            U (x) & \n = \n & \hbox{Poisson process with intensity one} \vspace*{4pt} \\
      C (x, y, z) & \n = \n & \hbox{Poisson process with intensity~$\rc / 4d^2$} & \hbox{when} & x \sim y \ \hbox{and} \ y \sim z \vspace*{4pt} \\
         D (x, y) & \n = \n & \hbox{Poisson process with intensity~$\rd / 2d$} & \hbox{when} & x \sim y. \end{array}
\end{equation}
 The process can then be constructed starting from any initial configuration as follows.
\begin{itemize}
 \item At the times of the Poisson process~$B (x, y)$, we draw an arrow~$y \to x$ to indicate that if~$y$ is occupied by a cooperator,
       respectively a defector, and~$x$ is empty, then~$x$ becomes occupied by a cooperator, respectively a defector. \vspace*{2pt}
 \item At the times of the Poisson process~$U (x)$, we put a cross~$\times$ at site~$x$ to indicate that if~$x$ is occupied then it becomes empty. \vspace*{2pt}
 \item At the times of the Poisson process~$C (x, y, z)$, we put a dot~$\bullet$ at site~$z$ and draw a~$\c$-arrow~$y \to x$ to indicate that if both~$y$
       and~$z$ are occupied by cooperators and~$x$ is empty, then~$x$ becomes occupied by a cooperator. \vspace*{2pt}
 \item At the times of the Poisson process~$D (x, y)$, we draw a~$\d$-arrow~$y \to x$ to indicate that if~$y$ is occupied by a defector and~$x$ is empty,
       then~$x$ becomes occupied by a defector.
\end{itemize}
 Figure~\ref{fig:harris} gives an illustration of the graphical representation with time going up.
\begin{figure}[t]
\centering
\scalebox{0.40}{\input{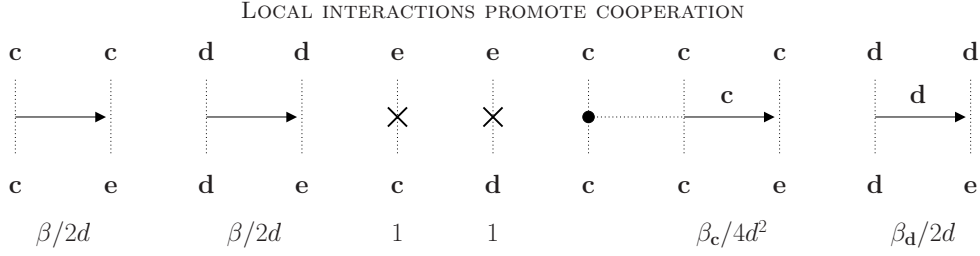}}
\caption{\upshape Graphical representation.}
\label{fig:harris}
\end{figure}
\begin{lemma} --
\label{lem:coupling}
 Assume~$\rcc \geq \rcd$ and~$\rdc \leq \rdd$, and for~$\is = \c, \d$ let
 $$ (\xi_t^{\is}) \ \hbox{be the process with parameters} \ \b, \rc = \rci \ \hbox{and} \ \rd = \rdi. $$
 Then, there is a coupling~$(\xi_t^{\c}, \xi_t^{\d})$ such that
 $$ \{x : \xi_t^{\c} (x) = \c \} \supset \{x : \xi_t^{\d} (x) = \c \} \quad \hbox{and} \quad
    \{x : \xi_t^{\c} (x) = \d \} \subset \{x : \xi_t^{\d} (x) = \d \} $$
 for all~$t > 0$ whenever the inclusions hold at time~$t = 0$.
\end{lemma}
\begin{proof}
 The basic idea is to construct the two processes from the common graphical representation depicted in Figure~\ref{fig:coupling-1} using six
 collections of independent Poisson processes.
\begin{itemize}
 \item The unlabeled arrows~$\to$, crosses~$\times$, $\c$-arrows and~$\d$-arrows have the same effect as before, and are used to construct both processes. \vspace*{2pt}
 \item The~$\c_+$-arrows have the same effect as the~$\c$-arrows on the first process but are not used in the construction of the second process. \vspace*{2pt}
 \item The~$\d_+$-arrows have the same effect as the~$\d$-arrows on the second process but are not used in the construction of the first process.
\end{itemize}
 The superposition property of Poisson processes, and the fact that
 $$ \rcd + (\rcc - \rcd) = \rcc \quad \hbox{and} \quad \rdc + (\rdd - \rdc) = \rdd $$
 imply that the two processes have the appropriate transition rates.
 In addition, it is straightforward to check from our construction that the coupling satisfies
 $$ (\xi_t^{\c} (x), \xi_t^{\d} (x)) \in \{(\c, \c), (\d, \d), (\e, \e), (\c, \d), (\c, \e), (\e, \d) \} \quad \hbox{for all} \quad x \in \Z^d $$
 whenever this is true at time~$t = 0$.
 This completes the proof.
\end{proof} \\ \\
 Using the coupling from the previous lemma, we are now ready to prove Theorem~\ref{th:monotonicity}. \\ \\
\begin{figure}[t]
\centering
\scalebox{0.40}{\input{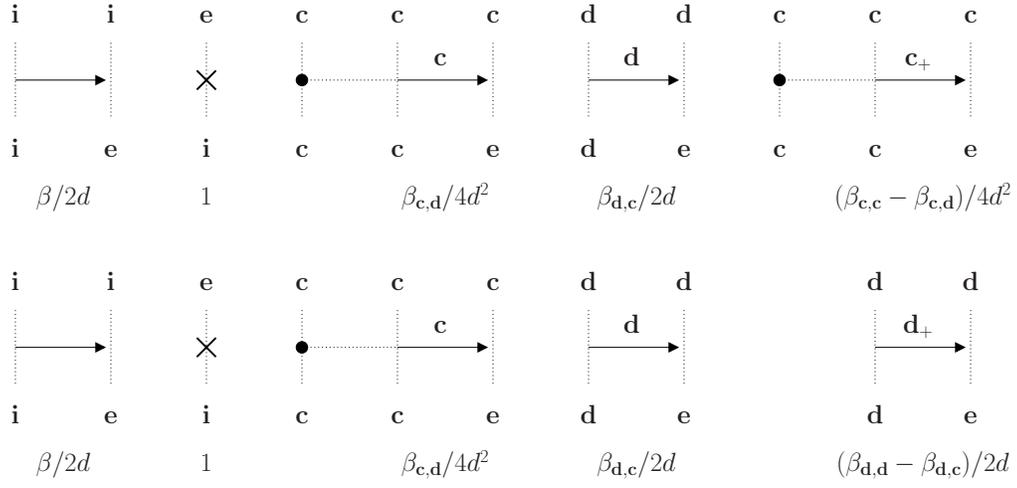}}
\caption{\upshape Graphical representation used in Lemma~\ref{lem:coupling}.}
\label{fig:coupling-1}
\end{figure}
\begin{proofof}{Theorem~\ref{th:monotonicity}}
 Let~$\rcc \geq \rcd$ and~$\rdc \leq \rdd$ and for~$\is \in \{\c, \d \}$ let
 $$ \begin{array}{rcl}
    \theta_{\c} (\rci, \rdi) & \n = \n & \hbox{survival probability of the cooperators for the} \vspace*{2pt} \\ &&
                                         \hbox{process with parameters} \ \b, \rc = \rci \ \hbox{and} \ \rd = \rdi \vspace*{4pt} \\
    \theta_{\d} (\rci, \rdi) & \n = \n & \hbox{survival probability of the defectors for the} \vspace*{2pt} \\ &&
                                         \hbox{process with parameters} \ \b, \rc = \rci \ \hbox{and} \ \rd = \rdi. \end{array} $$
 To prove the theorem, it suffices to prove that
 $$ \theta_{\c} (\rcc, \rdc) \geq \theta_{\c} (\rcd, \rdd) \quad \hbox{and} \quad \theta_{\d} (\rcc, \rdc) \leq \theta_{\c} (\rcd, \rdd). $$
 For each realization~$\xi_{\cdot} \in \{\c, \d, \e \}^{\mathbb R_+ \times \Z^d}$ of the process, let
 $$ h_{\c} (\xi_{\cdot}) = \ind \{\xi_t (x) = \c \ \hbox{i.o.} \} \quad \hbox{and} \quad h_{\d} (\xi_{\cdot}) = \ind \{\xi_t (x) = \d \ \hbox{i.o.} \}. $$
 Also, let~$(\xi_t^{\c}, \xi_t^{\d})$ be the coupling defined in Lemma~\ref{lem:coupling}.
 Then, for all~$t > 0$,
 $$ \{x : \xi_t^{\c} (x) = \c \} \supset \{x : \xi_t^{\d} (x) = \c \} \quad \hbox{and} \quad
    \{x : \xi_t^{\c} (x) = \d \} \subset \{x : \xi_t^{\d} (x) = \d \} $$
 therefore~$h_{\c} (\xi_{\cdot}^{\c}) \geq h_{\c} (\xi_{\cdot}^{\d})$ and~$h_{\d} (\xi_{\cdot}^{\c}) \leq h_{\d} (\xi_{\cdot}^{\d})$.
 Taking the expected value,
 $$ \begin{array}{rcl}
    \theta_{\c} (\rcc, \rdc) = P (\xi_t^{\c} (x) = \c \ \hbox{i.o.}) = E (h_{\c} (\xi_{\cdot}^{\c})) & \n \geq \n & E (h_{\c} (\xi_{\cdot}^{\d})) = \theta_{\c} (\rcd, \rdd) \vspace*{4pt} \\
    \theta_{\d} (\rcc, \rdc) = P (\xi_t^{\c} (x) = \d \ \hbox{i.o.}) = E (h_{\d} (\xi_{\cdot}^{\c})) & \n \leq \n & E (h_{\d} (\xi_{\cdot}^{\d})) = \theta_{\d} (\rcd, \rdd). \end{array} $$
 This completes the proof.
\end{proofof}


\section{The cooperators win (upper bound for~$\rc^+$)}
\label{sec:upper}

\indent The objective of this section is to prove that the cooperators win when~$\rc$ is sufficiently large but finite, which shows that the critical
 value~$\rc^+$ is finite.
 Our proof is based on a block construction, an approach/technique introduced in~\cite{bramson_durrett_1988} and reviewed in~\cite{durrett_1995}
 that consists in coupling the interacting particle system properly rescaled in space and time with oriented site percolation. \\


\noindent {\bf Block construction}.
 In the context of our cooperator-defector system, we will prove that whenever a space-time block is ``almost'' completely occupied by
 cooperators, a property referred to as~$\c$-occupied below, the~$2d$ adjacent blocks later in time satisfy the same property with a probability close to one.
 To define our construction rigorously, we first consider the lattice
 $$ \Lat = \{(z, n) \in \Z^d \times \N : z_1 + \cdots + z_d + n \ \hbox{is even} \}, $$
 which we turn into a directed graph~$\G$ by adding arrows
 $$ (z, n) \to (z', n') \quad \hbox{if and only if} \quad z' = z \pm e_j \ \hbox{for some} \ j = 1, 2, \ldots, d \ \ \hbox{and} \ \ n' = n + 1 $$
 where $e_j$ is the $j$th unit vector. Define
 $$ \B_z = z + \{-1, 0, 1 \}^d \quad \hbox{for all} \quad z \in \Z^d. $$
 Let~$T > 0$ to be fixed later, and consider the events
 $$ \begin{array}{rcl}
    \C (z, n) & \n : \n & \card \,\{x \in \B_z : \xi_t (x) = \c \} \geq 3^d - 1 \vspace*{4pt} \\  & \n \n & \hbox{and} \
                          \card \,\{x \in \B_z : \xi_t (x) = \d \} = 0 \quad \hbox{for all} \quad nT \leq t \leq (n + 1) \,T \end{array} $$
 for all~$(z, n) \in \Lat$.
 When this last event occurs, we declare site~$(z, n)$ to be~$\c$-occupied.
 To define oriented site percolation on the directed graph~$\G$, we let~$\ep > 0$ to be fixed later, and assume that each site of the lattice~$\Lat$
 is open with probability~$1 - \ep$ and closed with probability~$\ep$.
 We further assume that the process is seven-dependent, meaning that
 $$ P ((z_i, n_i) \ \hbox{is closed for} \ i = 1, 2, \ldots, m) = \ep^m $$
 whenever~$|z_i - z_j| \vee |n_i - n_j| > 7$ for~$i \neq j$.
 Recalling that a site is wet if and only if it can be reached from a directed path of open sites starting at level zero, i.e., starting at some
 open site~$(z, 0) \in \Lat$, the goal is to use a coupling argument to prove that, at each level~$n$, the set~$X_n$ dominates stochastically the
 set~$W_n^{\ep}$ where
 $$ X_n = \{z : (z, n) \ \hbox{is~$\c$-occupied} \} \quad \hbox{and} \quad W_n^{\ep} = \{z : (z, n) \ \hbox{is wet} \}. $$
 This can be done by thinking of the process as being generated from the graphical representation above, and constructing
 a collection of good events~$\A (z, n)$ that are measurable with respect to the graphical representation restricted to bounded space-time blocks,
 occur with probability at least~$1 - \ep$, and insure that the property of being~$\c$-occupied spreads in space and time. \\


\noindent {\bf The good events~$\A (z, n)$}.
 We now define the good events mentioned above and estimate their probability.
 These events can be conveniently written as the intersection of three events defined from the graphical representation introduced in the previous section. Let
 $$ \B_+ = (\B_{e_1} \cup \B_{- e_1}) \cup (\B_{e_2} \cup \B_{- e_2}) \cup \cdots \cup (\B_{e_d} \cup \B_{- e_d}) \quad \hbox{and} \quad \B_- = \B_+ \setminus \B_0, $$
 and consider the first event
 $$ \begin{array}{rcl}
      A_1 & \n : \n & \hbox{there is at least one death mark~$\times$ at each of} \vspace*{2pt} \\ & \n \n &
                      \hbox{the~$2d \times 3^{d - 1}$ sites in~$\B_-$ between time~$T$ and time~$2T$.} \end{array} $$
 The probability of this event is computed explicitly in the following lemma.
\begin{lemma} --
\label{lem:A1}
 We have~$P (A_1) = (1 - e^{-T})^{2d \times 3^{d - 1}}$.
\end{lemma}
\begin{proof}
 Let~$\tau_x$ be the first time after~$T$ there is a death mark at~$x$.
 Because~$\tau_x - T$ is exponentially distributed with parameter one, and Poisson processes at different sites are independent,
 $$ \begin{array}{rcl}
      P (A_1) & \n = \n & P (\tau_x < 2T \ \hbox{for all} \ x \in \B_-) \vspace*{4pt} \\
              & \n = \n & (1 - P (\tau_0 - T > T))^{\card \B_-} = (1 - e^{-T})^{2d \times 3^{d - 1}} \end{array} $$
 which completes the proof.
\end{proof} \\ \\
 To define the second event, we let~$T_0 = 0$ and, for every~$i \in \N^*$,
 $$ \begin{array}{rcl}
      T_i & \n = \n & \hbox{the~$i$th time there is either a death mark~$\times$ at a site in~$\B_+$} \vspace*{2pt} \\
          &         & \hbox{or a~$\d$-arrow pointing at a site in~$\B_+$}. \end{array} $$
 Letting~$K = \max \,\{i : T_i < 2T \}$, the second event is defined as
 $$ \begin{array}{rcl}
      A_2 & \n : \n & T_{i + 1} - T_i > 2 \delta \quad \hbox{for all} \quad i = 0, 1, \ldots, K - 1, \end{array} $$
 where~$\delta > 0$ will be fixed later.
\begin{lemma} --
\label{lem:A2}
 Let~$r = 3^{d - 1} (2d + 3)(\b + \rd + 1)$.
 There exists~$a > 0$ such that
 $$ P (A_2) \geq 1 - \exp (- aT) - 4rT (1 - e^{- 2 \delta r}) \quad \hbox{for all~$T$ large}. $$
\end{lemma}
\begin{proof}
 To begin with, we observe that
\begin{itemize}
 \item death marks appear independently at each vertex at rate one, \vspace*{2pt}
 \item $\d$-arrows pointing at a given vertex appear independently at rate~$\b + \rd$ and \vspace*{2pt}
 \item there are~$3^{d - 1} (2d + 3)$ vertices in~$\B_+$.
\end{itemize}
 Therefore the number of death marks and~$\d$-arrows by time~$2T$ is
 $$ K = \max \,\{i : T_i < 2T \} = \poisson (2rT) \quad \hbox{where} \quad r = 3^{d - 1} (2d + 3)(\b + \rd + 1). $$
 In particular, large deviations estimates for the Poisson distribution give the existence of a positive constant~$a > 0$ that only depends on~$r$ such that
 $$ P (K > 4rT) \leq \exp (- aT) \quad \hbox{for all~$T$ large}. $$
 Using also that the random variables~$T_{i + 1} - T_i$ are independent and exponentially distributed with parameter~$r$, we deduce that
 $$ \begin{array}{rcl}
      P (A_2) & \n \geq \n & 1 - P (K > 4rT) - P (A_2^c \,| \,K \leq 4rT) \vspace*{4pt} \\
              & \n \geq \n & 1 - \exp (- aT) - P (T_{i + 1} - T_i \leq 2 \delta \ \hbox{for some} \ i < 4rT) \vspace*{4pt} \\
              & \n \geq \n & 1 - \exp (- aT) - 4rT \,P (T_1 \leq 2 \delta) = 1 - \exp (- aT) - 4rT \left(1 - e^{- 2 \delta r} \right). \end{array} $$
 This completes the proof.
\end{proof} \\ \\
 Finally, consider the event
 $$ \begin{array}{rcl}
      A_3 & \n : \n & \hbox{for all~$n \leq 2T / \delta$ and all~$x \in \B_+$, there is at least} \vspace*{2pt} \\ & \n \n &
                      \hbox{ one~$\c$-arrow from a site in~$\B_0$ pointing at~$x$ beween times~$n \delta$ and~$(n + 1) \,\delta$}. \end{array} $$
\begin{lemma} --
\label{lem:A3}
 Let~$R = (\b + \rc / 2d) / 2d$. Then,
 $$ P (A_3) \geq (1 - e^{- \delta R})^{2 \times 3^{d - 1} (2d + 3) T / \delta}. $$
\end{lemma}
\begin{proof}
 For each site~$x \in \B_+$, let~$\sigma_x$ be the first time there is a~$\c$-arrow from a site in~$\B_0$ pointing at site~$x$.
 Because~$\sigma_x$ is exponentially distributed with parameter at least~$R$, Poisson processes at different sites are independent, and there
 are~$3^{d - 1} (2d + 3)$ vertices in~$\B_+$,
 $$ P (\sigma_x < \delta \ \hbox{for all} \ x \in \B_+) = \prod_{x \in \B_+} (1 - P (\sigma_x \geq \delta)) \geq
      (1 - e^{- \delta R})^{3^{d - 1} (2d + 3)}. $$
 By the memoryless property of the exponential distribution,
 $$ P (A_3) \geq \left((1 - e^{- \delta R})^{3^{d - 1} (2d + 3)} \right)^{2T / \delta}
            \geq (1 - e^{- \delta R})^{2 \times 3^{d - 1} (2d + 3) T / \delta} $$
 and the proof is complete.
\end{proof} \\ \\
 To define the collection of good events, we set~$\A (0, 0) = A_1 \cap A_2 \cap A_3$ and define~$\A (z, n)$ similarly through a
 translation of vector~$(z, nT)$.
\begin{lemma} --
\label{lem:A123}
 For all~$\ep, \b, \rd > 0$, there exist~$T, \delta > 0$ such that
 $$ P (\A (z, n)) = P (A_1 \cap A_2 \cap A_3) \geq 1 - \ep \quad \hbox{for all~$(z, n) \in \Lat$ and all~$\rc$ large}. $$
\end{lemma}
\begin{proof}
 To begin with, we observe that, because (the distribution of) the graphical representation is translation invariant in space and time
 $$ P (\A (z, n)) = P (\A (0, 0)) = P (A_1 \cap A_2 \cap A_3) \quad \hbox{for all~$(z, n) \in \Lat$}, $$
 which proves the equality in the statement of the lemma.
 To prove the inequality, we first fix the scale parameter~$T < \infty$ large such that
\begin{equation}
\label{eq:A123-2}
  (1 - e^{-T})^{2d \times 3^{d - 1}} \geq 1 - \ep / 3 \qquad \hbox{and} \qquad \exp (- aT) \leq \ep / 6
\end{equation}
 where~$a > 0$ is the constant in the statement of Lemma~\ref{lem:A2}.
 The scale parameter~$T < \infty$ being fixed, we can fix the parameter~$\delta > 0$ small such that
\begin{equation}
\label{eq:A123-3}
  4rT (1 - e^{- 2 \delta r}) \leq \ep / 6.
\end{equation}
 Finally, the scale parameters~$T < \infty$ and~$\delta > 0$ being fixed, because the rate~$R$ defined in Lemma~\ref{lem:A3} goes to infinity
 as~$\rc \to \infty$, for all rate~$\rc$ sufficiently large, we have
\begin{equation}
\label{eq:A123-4}
  (1 - e^{- \delta R})^{2 \times 3^{d - 1} (2d + 3) T / \delta} \geq 1 - \ep / 3.
\end{equation}
 Combining~\eqref{eq:A123-2}--\eqref{eq:A123-4} and Lemmas~\ref{lem:A1}--\ref{lem:A3} implies that
\begin{equation}
\label{eq:A123-5}
  \begin{array}{rcl}
    P (A_1 \cap A_2 \cap A_3) & \n \geq \n & 1 - (1 - P (A_1)) - (1 - P (A_2)) - (1 - P (A_3)) \vspace*{4pt} \\
                              & \n = \n & - 2 + P (A_1) + P (A_2) + P (A_3) \vspace*{4pt} \\
                              & \n \geq \n & - 2 + 3 \,(1 - \ep / 3) = 1 - \ep. \end{array}
\end{equation}
 This completes the proof.
\end{proof} \\


\noindent {\bf Coupling with oriented percolation}.
 With Lemma~\ref{lem:A123} in hands, we can now prove the existence of a coupling in which the set of~$\c$-occupied sites dominates stochastically
 the set of wet sites in the oriented site percolation model.
 The basic idea is to use the previous estimates to check that the assumptions of~\cite[Theorem~4.3]{durrett_1995} are satisfied, which is done
 in the next lemma.
\begin{lemma} --
\label{lem:perco}
 For all~$\ep, \b, \rd > 0$, there exist~$T, \delta > 0$ such that, for all~$\rc$ large,
 $$ P (z \in W_n^{\ep}) \leq P (z \in X_n) \ \hbox{for all} \ (z, n) \in \Lat \quad \hbox{whenever} \quad W_0^{\ep} \subset X_0. $$
\end{lemma}
\begin{proof}
 The first step is to show how the events~$\C (z, n)$ are related to the good events~$\A (z, n)$ introduced above.
 To begin with, observe that, on the event~$A_2 \cap A_3$,
\begin{itemize}
 \item For all~$i = 0, 1, \ldots, K - 1$ and all~$x \in \B_+$, there is at least one~$\c$-arrow from a site in~$\B_0$ pointing at site~$x$
       between time~$T_i$ and time~$T_{i + 1}$.
\end{itemize}
 Assuming in addition that~$\C (0, 0)$ occurs gives the following:
\begin{itemize}
 \item Up to time~$2T$, each time a site in~$\B_+$ becomes empty, it becomes occupied by a cooperator before any other site in~$\B_+$
       becomes empty.
\end{itemize}
 Assuming also that~$A_1$ occurs, which guarantees that all sites in~$\B_-$ become empty at least once between time~$T$ and time~$2T$, we get
 $$ \begin{array}{rcl}
    \A (0, 0) \cap \{0 \in X_0 \} & \n = \n & (A_1 \cap A_2 \cap A_3) \cap \C (0, 0) \vspace*{4pt} \\
                                  & \n \subset \n & \C (e_1, 1) \cap \C (- e_1, 1) \cap \cdots \cap \C (e_d, 1) \cap \C (- e_d, 1) \vspace*{4pt} \\
                                  & \n = \n & \{\pm e_j \in X_1 \ \hbox{for} \ j = 1, 2, \ldots, d \}. \end{array} $$
 Using also translation invariance, we have more generally
\begin{equation}
\label{eq:perco-1}
   \A (z, n) \cap \{z \in X_n \} \subset \{z \pm e_j \in X_{n + 1} \ \hbox{for} \ j = 1, 2, \ldots, d \} \quad \hbox{for all} \quad (z, n) \in \Lat.
\end{equation}
 From Lemma~\ref{lem:A123}, we also have that~$T$ and~$\delta$ can be chosen such that
\begin{equation}
\label{eq:perco-2}
  P (\A (z, n)) \geq 1 - \ep \quad \hbox{for all~$(z, n) \in \Lat$ and all~$\rc$ large}.
\end{equation}
 Observe also from the definition that
 $$ \begin{array}{rl}
      A_1 \cap A_2 \cap A_3 & \hbox{is measurable with respect to the graphical representation} \vspace*{2pt} \\
                            & \hbox{restricted to the space-time region~$\R (0, 0) = \{-3, \ldots, 3 \}^d \times [0, 2T]$}. \end{array} $$
 Because~$\A (z, n)$ is defined from a translation of~$\A (0, 0)$,
\begin{equation}
\label{eq:perco-3}
  \begin{array}{rl}
    \A (z, n) & \hbox{is measurable with respect to the graphical representation restricted} \vspace*{2pt} \\
              & \hbox{to the space-time region~$\R (z, n) = (z + \{- 3, \ldots, 3 \}^d) \times [nT, (n + 2) T]$}. \end{array}
\end{equation}
 The three properties~\eqref{eq:perco-1}--\eqref{eq:perco-3} are the assumptions of~\cite[Theorem~4.3]{durrett_1995}, which implies that there is a coupling of
 the interacting particle system and seven-dependent oriented site percolation process in which sites are open with probability~$1 - \ep$ such that
 $$ P (W_n^{\ep} \subset X_n) = 1 \quad \hbox{whenever} \quad W_0^{\ep} \subset X_0, $$
 where the seven-dependency follows from the fact that
 $$ \R (z, n) \cap \R (z', n') = \varnothing \quad \hbox{whenever} \quad |z - z'| \vee |n - n'| > 7. $$
 The lemma directly follows from the existence of this coupling.
\end{proof} \\


\noindent {\bf Proof of the upper bound}.
 We are now ready to prove that~$\rc^+ (\b, \rd) < \infty$, i.e., starting from a translation invariant product measure, the cooperators win provided
 the benefit from cooperation~$\rc$ is sufficiently large.
 We start by proving survival of the cooperators.
\begin{lemma} --
\label{lem:invade}
 For all~$\b > \b_*$ and~$\rd \geq 0$,
 $$ P (\xi_t (x) = \c \ \hbox{i.o.}) = 1 \quad \hbox{for all~$\rc$ large}. $$
\end{lemma}
\begin{proof}
 A contour argument~\cite{durrett_1984} implies that there exists~$\ep > 0$ small such that
 $$ P (z \in W_n^{\ep} \ \hbox{i.o.}) = 1 \quad \hbox{whenever} \quad \card (W_0^{\ep}) = \infty. $$
 Then, using Lemma~\ref{lem:perco} and that~$\card (X_0) = \infty$ for any initial translation invariant product measure with a positive
 density of cooperators, we deduce that there is~$\rc < \infty$ large such that
 $$ P (\xi_t (x) = \c \ \hbox{i.o.}) \geq P (z \in X_n \ \hbox{i.o.}) \geq P (z \in W_n^{\ep} \ \hbox{i.o.}) = 1, $$
 where the first inequality follows from the definition of a~$\c$-occupied site.
 In particular, the cooperators survive for~$\rc$ large and the lemma follows.
\end{proof} \\ \\
 Because of the presence of closed sites in the percolation process that may result in the presence of space-time regions containing defectors
 even at equilibrium, the previous lemma does not exclude the possibility that cooperators and defectors coexist.
 To also prove extinction of the defectors, we rely on the lack of percolation of the dry sites when~$\ep > 0$ is small.
\begin{lemma} --
\label{lem:dry}
 For all~$\b > \b_*$ and~$\rd \geq 0$,
 $$ P (\xi_t (x) = \d \ \hbox{i.o.}) = 0 \quad \hbox{for all~$\rc$ large}. $$
\end{lemma}
\begin{proof}
 Let~$\H$ be the directed graph with the same vertex set~$\Lat$ as the directed graph~$\G$ but with the additional horizontal arrows:
 $$ (z, n) \to (z', n) \quad \hbox{if} \quad z' = z \pm 2e_j \ \hbox{for some} \ j = 1, 2, \ldots, d. $$
 Writing~$(w, 0) \to_{\G} (z, n)$ to indicate the presence of a dry path in~$\G$, i.e., a directed path in~$\G$ of sites that are not wet,
 Durrett~\cite{durrett_1992} proved that, when~$\ep > 0$ is sufficiently small and sites are open with probability~$1 - \ep$, dry sites
 do not percolate, which implies that
 $$ \begin{array}{l} \lim_{n \to \infty} P ((w, 0) \to_{\G} (z, n) \ \hbox{for some} \ w \in \Z) = 0 \end{array} $$
 for all~$z \in \Z$.
 Note that a dry path in~$\G$ is also a dry path in~$\H$ but the converse is false because the latter has more oriented edges than the former.
 However, Lemma~9 in~\cite{lanchier_2013} shows that the previous result can be extended to dry paths in~$\H$, i.e.,
\begin{equation}
\label{eq:dry-1}
 \begin{array}{l} \lim_{n \to \infty} P ((w, 0) \to_{\H} (z, n) \ \hbox{for some} \ w \in \Z) = 0 \end{array}
\end{equation}
 where~$(w, 0) \to_{\H} (z, n)$ indicates the presence of a dry path in~$\H$.
 Now, observe that, because defectors cannot appear spontaneously,
\begin{equation}
\label{eq:dry-2}
  \begin{array}{l}
  \xi_t (x) = \d \ \hbox{for some} \ (x, t) \in \B_z \times [nT, (n + 1) T] \vspace*{4pt} \\ \hspace*{40pt}
  \hbox{implies that} \quad (w, 0) \to_{\H} (z, n) \ \hbox{for some} \ w \in \Z. \end{array}
\end{equation}
 Combining~\eqref{eq:dry-1} and~\eqref{eq:dry-2}, we deduce that
 $$ \begin{array}{l} P (\xi_t (x) = \d \ \hbox{i.o.}) \leq \lim_{n \to \infty} P ((w, 0) \to_{\H} (z, n) \ \hbox{for some} \ w \in \Z) = 0. \end{array} $$
 This completes the proof.
\end{proof} \\ \\
 Combining the previous two lemmas shows that the cooperators win for all~$\rc < \infty$ sufficiently large, which is equivalent to~$\rc^+ (\b, \rd) < \infty$.


\section{The defectors win (lower bound for~$\rc^-$)}
\label{sec:lower}

\indent In this section, we prove that
\begin{equation}
\label{eq:critical}
  \rc^- (\b, \rd) > 2d \rd / (2d - 1) > \rd \quad \hbox{for all} \quad \b > \b_* \ \hbox{and} \ \rd > 0,
\end{equation}
 meaning that there exists~$\rc > 2d \rd / (2d - 1)$ such that the defectors win.
 This, together with the result from the previous section, will complete the proof of Theorem~\ref{th:critical}.
 To begin with, observe that the rate at which an empty site, say~$x$, becomes occupied by a defector is
 $$ \sum_{y \sim x} \bigg(\frac{\b + \rd}{2d} \bigg) $$
 while the rate at which it becomes occupied by a cooperator is
 $$ \sum_{y \sim x} \bigg(\frac{\b}{2d} + \sum_{z \sim y} \ \frac{\rc \,\ind \{\xi (z) = \c \}}{4d^2} \bigg) \leq
    \sum_{y \sim x} \bigg(\frac{\b}{2d} + \sum_{z \sim y} \ \frac{\rc}{4d^2} \bigg) = \sum_{y \sim x} \bigg(\frac{\b + \rc}{2d} \bigg). $$
 In particular, the process can be coupled with a multi-type contact process in which type~1 particles give birth at rate~$\b + \rc$
 and type~2 particles give birth at rate~$\b + \rd$ in such a way that type~1 particles dominate the cooperators and type~2 particles are dominated by the
 defectors, provided this is true initially.
 Starting from a translation invariant product measure with a positive density of each type, Theorem~1 in~\cite{neuhauser_1992} states that type~2 particles
 win whenever
 $$ \b + \rc < \b + \rd \quad \hbox{and} \quad \b > \b_*. $$
 This and the coupling imply that, when~$\rc > \rd$, the defectors win, so~$\rc^- (\b, \rd) \geq \rd$.
 Proving the bound in~\eqref{eq:critical} is more difficult.
 The idea is to show that the defectors again win when
 $$ \rc = 2d \rd / (2d - 1) > 0 \quad \hbox{and} \quad \b > \b_* $$
 using a block construction to compare the process properly rescaled in space and time with oriented site percolation as in the previous section.
 The block construction allows for the use of a perturbation argument to deduce that the defectors still win for some~$\rc > 2d \rd / (2d - 1)$. \\


\noindent {\bf Graphical representation}.
 First, we construct the process with
\begin{equation}
\label{eq:equal}
  \rc = 2d \rd / (2d - 1) > 0
\end{equation}
 by coupling the births and deaths of the cooperators and defectors.
 The process can be constructed from the first three collections of Poisson processes in~\eqref{eq:harris} as follows.
\begin{itemize}
 \item At the times of the Poisson process~$B (x, y)$, we draw an arrow~$y \to x$ to indicate that if~$y$ is occupied by a cooperator,
       respectively a defector, and~$x$ is empty, then~$x$ becomes occupied by a cooperator, respectively a defector. \vspace*{2pt}
 \item At the times of the Poisson process~$U (x)$, we put a cross~$\times$ at site~$x$ to indicate that if~$x$ is occupied then it becomes empty. \vspace*{2pt}
 \item At the times of the Poisson process~$C (x, y, z)$ with~$z \neq x$, we put a dot~$\bullet$ at site~$z$ and draw an arrow~$y \to x$ to indicate that if both
       site~$y$ and site~$z$ are occupied by cooperators and site~$x$ is empty, then site~$x$ becomes occupied by a cooperator, while if site~$y$ is occupied by
       a defector and site~$x$ is empty, then site~$x$ becomes occupied by a defector. \vspace*{2pt}
 \item At the times of the Poisson process~$C (x, y, z)$ with~$z = x$, we put a dot~$\bullet$ at site~$z$ and draw a~$\c$-arrow~$y \to x$ to indicate that if
       both site~$y$ and site~$z$ are occupied by cooperators and site~$x$ is empty, then site~$x$ becomes occupied by a cooperator.
\end{itemize}
 See Figure~\ref{fig:coupling-2} for a picture.
\begin{figure}[t]
\centering
\scalebox{0.40}{\input{coupling-2.pstex_t}}
\caption{\upshape Graphical representation with~$\rc = 2d \rd / (2d - 1)$.}
\label{fig:coupling-2}
\end{figure}
 Note that, for each pair of neighbors~$x \sim y$,
 $$ \card \,\{z \in \Z^d : z \sim y \ \hbox{and} \ z \neq x \} \,(\rc / 4d^2) = (2d - 1)(\rc / 4d^2) = \rd / 2d. $$
 In particular, the superposition property for independent Poisson processes implies that there is an unlabeled dot-arrow~$y \to x$ at the times of a Poisson
 process with intensity~$\rd / 2d$, therefore the construction indeed produces the desired birth rates.
 To prove that the defectors win when~\eqref{eq:equal} holds, we now show that, regardless of the initial configuration, the cooperators cannot give birth
 through the~$\c$-arrows or some of the unlabeled dot-arrows.
\begin{lemma} --
\label{lem:remove-arrow}
 The~$\c$-arrows have no effect.
\end{lemma}
\begin{proof}
 Let~$s \in C (x, y, z)$ with~$z = x$ be the arrival time of a~$\c$-arrow~$y \to x$.
 Then, we have the following alternative for the configuration just before time~$s$.
\begin{itemize}
 \item Site~$x$ is occupied by a defector or a cooperator.
       In this case, any potential birth through the~$\c$-arrow is suppressed so the state at site~$x$ remains the same. \vspace*{2pt}
 \item Both site~$x$ and site~$y$ are empty.
       In this case, because site~$y$ is empty, no individual gives birth through the~$\c$-arrow so site~$x$ remains empty. \vspace*{2pt}
 \item Site~$x$ is empty and site~$y$ is occupied by a defector.
       In this case, because defectors cannot give birth through a~$\c$-arrow, site~$x$ remains empty. \vspace*{2pt}
 \item Site~$x$ is empty and site~$y$ is occupied by a cooperator.
       In this case, site~$x$ becomes occupied by a cooperator if and only if site~$z$ is occupied by a cooperator just before time~$s$.
       But we have~$z = x$ so site~$z$ is empty just before time~$s$ and site~$x$ remains empty.
\end{itemize}
 In either case, the state at site~$x$ remains unchanged at time~$s$.
\end{proof} \\ \\
 To state our next results, for each time~$s \in C (x, y, z)$, we let
 $$ \begin{array}{rcl}
     u (s) & \n = \n & \sup \,\{t < s : t \in U (z) \} \vspace*{4pt} \\
     v (s) & \n = \n & \sup \,\{t < s : t \in B (z, y') \ \hbox{for some} \ y' \sim z \ \hbox{or} \ t \in C (z, y', z') \ \hbox{for some} \ y' \sim z, z' \sim y' \} \end{array} $$
 denote respectively the last time before time~$s$ there is a death mark~$\times$ at site~$z$ and the last time before time~$s$ there is either an arrow or a dot-arrow
 pointing at site~$z$.
 Then, we say that the dot-arrow~$y \to x$ at time~$s$ is sterile whenever
 $$ s - u (s) < 1 < s - v (s) < 2. $$
\begin{lemma} --
\label{lem:proba-sterile}
 Let~$s_i \in C (x_i, y_i, z_i)$ for~$i = 1, 2, \ldots, n$.
 Then the dot-arrows~$y_i \to x_i$ at times~$s_i$ are independently sterile with the same probability
 $$ (1 - e^{-1})(1 - e^{- (\b + \rc)}) \,e^{- (\b + \rc)} > 0 $$
 whenever~$|x_i - x_j| \vee |s_i - s_j| > 2$ and~$s_i > 2$ for all~$i \neq j$.
\end{lemma}
\begin{proof}
 The fact that the~$n$ events in the statement are independent follows from the fact that disjoint parts of the graphical representation are independent.
 The fact that they also have the same probability follows from the fact that (the distribution of) the graphical representation is translation invariant
 in space and time.
 In addition, assuming that the Poisson processes used in the graphical representation are defined for negative times,
 $$ s - u (s) = \exponential (1) \quad \hbox{and} \quad s - v (s) = \exponential (\b + \rc) $$
 and the two random variables are independent.
 It follows that
 $$ \begin{array}{l}
     P (s - u (s) < 1 < s - v (s) < 2) = P (s - u (s) < 1) \,P (1 < s - v (s) < 2) \vspace*{4pt} \\ \hspace*{25pt}
     = (1 - e^{-1})(e^{- (\b + \rc)} - e^{- 2 (\b + \rc)}) = (1 - e^{-1})(1 - e^{- (\b + \rc)}) \,e^{- (\b + \rc)} > 0. \end{array} $$
 This completes the proof.
\end{proof}
\begin{lemma} --
\label{lem:effect-sterile}
 Cooperators cannot give birth through a sterile arrow.
\end{lemma}
\begin{proof}
 Let~$s \in C (x, y, z)$ be the time of a sterile dot-arrow~$y \to x$.
 By definition, there is a death mark at site~$z$ at time~$ u (s) < s$ and no arrow pointing at site~$z$ between time~$u (s)$ and time~$s$, therefore
 site~$z$ must be empty at time~$s$.
 It follows that, even if site~$y$ is occupied by a cooperator at time~$s$, this cooperator cannot give birth through the dot-arrow.
\end{proof} \\ \\
 Motivated by Lemma~\ref{lem:remove-arrow}, we remove all the~$\c$-arrows from the graphical representation.
 Motivated by Lemma~\ref{lem:effect-sterile}, we also label all the sterile dot-arrows with a~$\d$ to indicate that only the defectors can give birth
 through these arrows. \\


\noindent {\bf Duality relationship}.
 The fact that the defectors win when~\eqref{eq:equal} holds follows from duality techniques in~\cite{durrett_neuhauser_1997} for the multi-type
 contact process.
 The multi-type contact process is associated to a so-called dual process in such a way that the state of each space-time point can be determined from
 the initial configuration and the structure of the dual process.
 For our cooperator-defector system, we define a process that we again call the dual process.
 This process, however, only gives a partial information of the state of each space-time point.
 Given a realization of the graphical representation above where the~$\c$-arrows have been removed and the sterile dot-arrows have been labeled with a~$\d$,
 we define paths, dual paths and dual process as follows.
\begin{definition}[dual path] --
 There is a path~$(y, s) \uparrow (x, t)$ if there are
 $$ \hbox{times} \ \ s = s_0 < s_1 < \cdots < s_{n + 1} = t \quad \hbox{and} \quad \hbox{sites} \ \ y = x_0, x_1, \ldots, x_n = x $$
 such that the following two conditions hold:
\begin{enumerate}
 \item For $i = 1, 2, \ldots, n$, there is an arrow from $x_{i - 1}$ to $x_i$ at time $s_i$ and \vspace{2pt}
 \item For $i = 0, 1, \ldots, n$, the vertical segments $\{x_i \} \times (s_i, s_{i + 1})$ do not contain any $\times$'s.
\end{enumerate}
 There is a dual path~$(x, t) \downarrow (y, t - s)$ if there is a path~$(y, t - s) \uparrow (x, t)$.
\end{definition}
\begin{definition}[dual process] --
 The dual process starting at~$(x, t)$ is defined as
 $$ \hat \xi_s^{(x, t)} = \{y \in \Z^d : (x, t) \downarrow (y, t - s) \} \quad \hbox{for all} \quad 0 \leq s \leq t. $$
\end{definition}
 There might be several dual paths connecting two space-time points, so strictly speaking the state of the dual process
 is not a subset but a multi-subset of~$\Z^d$ where each site is assigned a multiplicity representing a number of dual paths.
 The dual process can be visualized by injecting a fluid at~$(x, t)$ that flows down (backward in time), is stopped at the death marks
 and crosses the arrows in the opposite direction, as shown in Figure~\ref{fig:dual}.
 The dual process keeps track of all the potential ancestors of the individual (if any) at site~$x$ at time~$t$.
 The space-time set filled with the fluid, namely
 $$ \Gamma = \{(\hat \xi_s^{(x, t)}, s) : 0 \leq s \leq t \}, $$
 exhibits a tree structure that induces an ancestor hierarchy corresponding to the order in which ancestors determine the type of~$(x, t)$.
 The ancestor hierarchy for our model is the same as for the multi-type contact process and was first described in~\cite{neuhauser_1992},
 but we follow an idea in~\cite{lanchier_neuhauser_2006} to give a more rigorous definition.
 The hierarchy can be defined using a function
 $$ \phi : \Gamma \longrightarrow \mathfrak S $$
 that maps the tree structure into the set~$\mathfrak S$ of integer-valued sequences equipped with the usual lexicographic order~$\ll$ defined as
 $$ u \ll v \quad \hbox{if and only if} \quad u_i = v_i \ \ \hbox{for} \ \ i = 1, 2, \ldots, n - 1 \quad \hbox{and} \quad u_n < v_n $$
 for some integer $n \geq 1$.
 For all~$u \in \mathfrak S$, it is convenient to identify
 $$ u \equiv (u_1, u_2, \ldots, u_n) \quad \hbox{whenever} \quad u_n \neq 0 \ \ \hbox{and} \ \ u_{n + 1} = u_{n + 2} = u_{n + 3} = \cdots = 0. $$
 The function~$\phi$ is defined inductively as follows.
\begin{itemize}
 \item We start at~$\phi (x, 0) = (1, 0, 0, \ldots) \equiv 1$. \vspace*{2pt}
 \item Assume that~$y$ is added to the dual process at dual time~$s_*$ and that
       $$ \phi (y, s_*) = u \equiv (u_1, u_2, \ldots, u_n). $$
  \begin{itemize}
   \item[$\circ$] Go down the graphical representation starting at~$(y, t - s_*)$, let~$t - s_0$ be the first time a death mark~$\times$ is
                  encountered at site~$y$, then set
                  $$ \phi (y, s) = \phi (y, s_*) = u \quad \hbox{for all} \quad s_* \leq s < s_0. $$
   \item[$\circ$] Go back up from~$(y, t - s_0)$ to~$(y, t - s_*)$, let
                  $$ t - s_1 < t - s_2 < \cdots < t - s_m $$
                  be the times at which we encounter the tip of an arrow and denote the sites these~$m$ arrows originate from by~$y_1, y_2, \ldots, y_m$, respectively.
                  Then we set
                  $$ \phi (y_i, s_i) = (u_1, u_2, \ldots, u_n, i) \quad \hbox{for all} \quad i = 1, 2, \ldots, m. $$
  \end{itemize}
\end{itemize}
 At a given dual time~$s$, the ancestor hierarchy is described by
 $$ \hbox{$y$ comes before~$z$ in the hierarchy} \quad \hbox{if and only if} \quad \phi (y, s) \ll \phi (z, s). $$
 Figure~\ref{fig:dual} gives an example of a realization of the graphical representation along with the dual process drawn in thick solid and dashed lines.
 The solid lines keep track of the position of the first ancestor and the numbers at the bottom of the picture show the ancestor hierarchy.
 Note that there are two dual paths~$(x, t) \downarrow (0, 0)$ therefore site~0 appears twice in the ancestor hierarchy.
 It is both the second and the fourth ancestors.
 In our example, the ancestor hierarchy is
 $$ \hat \xi_t^{(x, t)} = (-1, 0, -3, 0, 2, 3). $$
 The dual process of our cooperator-defector system is similar to the dual process of the multi-type contact process~\cite{neuhauser_1992}.
 However, for the multi-type contact process, the state at~$(x, t)$ can be determined from the dual process and the initial
 configuration, whereas for our model, the dual process only gives a partial information.
 In our example, we have the following alternative.
\begin{itemize}
 \item[1.] All the ancestors are (initially) empty, in which case~$(x, t)$ is empty. \vspace*{2pt}
 \item[2.] At least one of the ancestors is occupied and the first occupied ancestor is occupied by a defector, in which case~$(x, t)$ is occupied by
           a defector. \vspace*{2pt}
 \item[3.] Items 1 and 2 do not hold.
           In this case, the first three ancestors are irrelevant because if any of these three ancestors is occupied by a cooperator then the cooperator
           is blocked by the~$\d$-arrow on the left of the picture on its way up to~$(x, t)$ and none of the three ancestors can be the actual ancestor of~$(x, t)$.
           Then, we have the following alternative.
 \begin{itemize}
  \item[$\circ$] The fourth ancestor, site~0, is occupied, in which case~$(x, t)$ is of this type.
  \item[$\circ$] The fourth ancestor is empty and the next occupied ancestor (if any) is occupied by a defector, in which case~$(x, t)$ is occupied by a defector.
  \item[$\circ$] The fourth ancestor is empty and the next occupied ancestor (if any) is occupied by a cooperator.
                 In this case, the cooperator meets the tail of a dot-arrow on its way up so the outcome depends on the state at the dot.
                 In particular, the dual process alone does not provide enough information to determine the state at~$(x, t)$.
 \end{itemize}
\end{itemize}
 Even though the dual process we defined does not always provide enough information to deduce the state at~$(x, t)$ from the initial configuration, the presence
 of~$\d$-arrows that block the cooperators while the defectors can give birth through all the arrows is sufficient to prove that the probability that~$(x, t)$ is
 occupied by a cooperator tends to zero as~$t \to \infty$. \\
\begin{figure}[t]
\centering
\scalebox{0.40}{\input{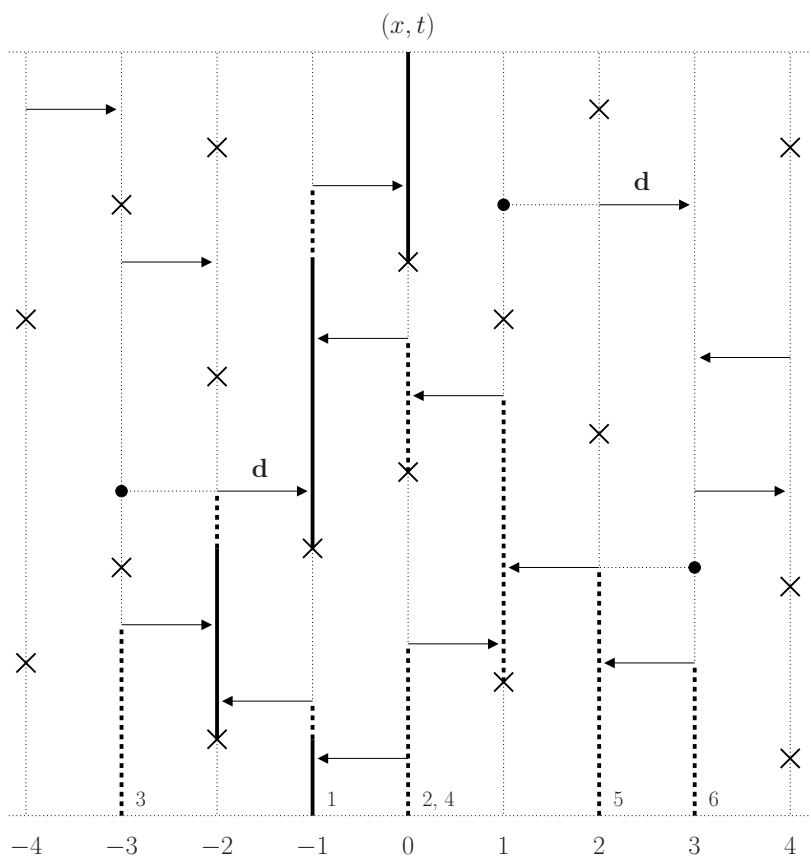}}
\caption{\upshape Dual process.}
\label{fig:dual}
\end{figure}


\noindent {\bf Block construction}.
 The dual process of the cooperator-defector system is similar to the dual process of the multi-type contact process in which the defectors have the
 same death rate as the cooperators but a higher birth rate.
 There is a slight difference in the structure.
\begin{itemize}
 \item In the multi-type contact process, all the arrows are independently turned into~$\d$-arrows with a fixed probability whereas the dot-arrows in
       the cooperator-defector system are turned into~$\d$-arrows based on the nearby graphical representation.
\end{itemize}
 There is also a difference in the interpretation.
\begin{itemize}
 \item For the multi-type contact process, the type of~$(x, t)$ can be determined from the type of the ancestors and the structure of the dual process.
       Because the cooperators cannot give birth through an unlabeled dot-arrow when the site marked with the dot is not occupied by a cooperator, for the
       cooperator-defector system, the initial configuration and the dual process are not always sufficient to determine the state at~$(x, t)$.
       However, the initial configuration and the dual process being fixed, if~$(x, t)$ is occupied by a defector in the multi-type contact process then
       it is also occupied by a defector for the cooperator-defector system.
\end{itemize}
 Neuhauser~\cite{neuhauser_1992} proved that the path of the first ancestor, which is the same for both processes, can be broken into independent and
 identically distributed pieces at some renewal points using an idea of Kuczek~\cite{kuczek_1989}.
 This idea was again used in~\cite{durrett_neuhauser_1997} in combination with a so-called repositioning algorithm to prove that, when properly rescaled
 in space and time, the multi-type contact process can be coupled with oriented site percolation to prove that the type with the higher birth rate expands
 linearly.
 Their proof is rather lengthy but easily adapts to our cooperator-defector system so we just state their result.
 Even though they focused on the~$d = 2$ case because this was the case of interest in their work, their argument holds in any spatial dimensions.
 As previously, let
 $$ \Lat = \{(z, n) \in \Z^d \times \N : z_1 + \cdots + z_d + n \ \hbox{is even} \}. $$
 which we turn into a directed graph using the same construction as before.
 Let~$L$ be a large integer to be fixed later, and introduce the spatial boxes
 $$ \B_z = Lz + [-L, L]^d \quad \hbox{for all} \quad z \in \Z^d. $$
 We further partition each~$\B_z$ by setting
 $$ \D_w = L^{0.1} w + (- L^{0.1} / 2, L^{0.1} / 2]^d \quad \hbox{and} \quad I_z = \{w \in \Z^d : \D_w \subset \B_z \}. $$
 Define also the collection of space-time rectangles
 $$ \B (z, n) = \B_z \times (nT + [0, T)) \quad \hbox{for all} \quad (z, n) \in \Lat $$
 where~$T = L^2$.
 Then, we consider the two collections of events
 $$ \begin{array}{rcl}
    \C_- (z, n) & \n : \n & \xi_t (x) \neq \c \ \hbox{for all} \ (x, t) \in \B (z, n) \vspace*{4pt} \\
    \C_+ (z, n) & \n : \n & \hbox{there is~$x \in \D_w$ such that~$\xi_{nT} (x) = \d$ for all~$w \in I_z$} \vspace*{2pt} \\ & \n \n &
                            \hbox{and} \ \xi_{nT} (x) \neq \c \ \hbox{for all} \ x \in \B_z \end{array} $$
 for all~$(z, n) \in \Lat$.
 In words, the rectangle~$\B (z, n)$ is void of cooperators, and each~$\D_w \subset \B_z$ is void of cooperators and has at least one
 defector at time~$nT$.
\begin{proposition} --
\label{prop:block}
 Assume~$\b > \b_*$ and~$\rc = 2d \rd / (2d - 1) > 0$.
 For all~$\ep > 0$, there exists a collection of events~$\A (z, n)$ that are measurable with respect to the graphical representation restricted to the
 space-time region
 $$ \R (z, n) = (Lz, nT) + ([- 3L, 3L]^d \times [0, 2T]) $$
 and such that, for all~$L$ sufficiently large,
 $$ \A (z, n) \cap \C_+ (z, n) \subset \C_+ (z \pm e_i, n + 1) \cap \C_- (z \pm e_i, n + 1) $$
 for all~$i = 1, 2, \ldots, d$ and~$P (\A (z, n)) \geq 1 - \ep$.
\end{proposition}
 See Propositions~3.1 and~3.2 in~\cite{durrett_neuhauser_1997} for a proof, and Figure~\ref{fig:block} for an illustration. \\
\begin{figure}[t]
\centering
\scalebox{0.44}{\input{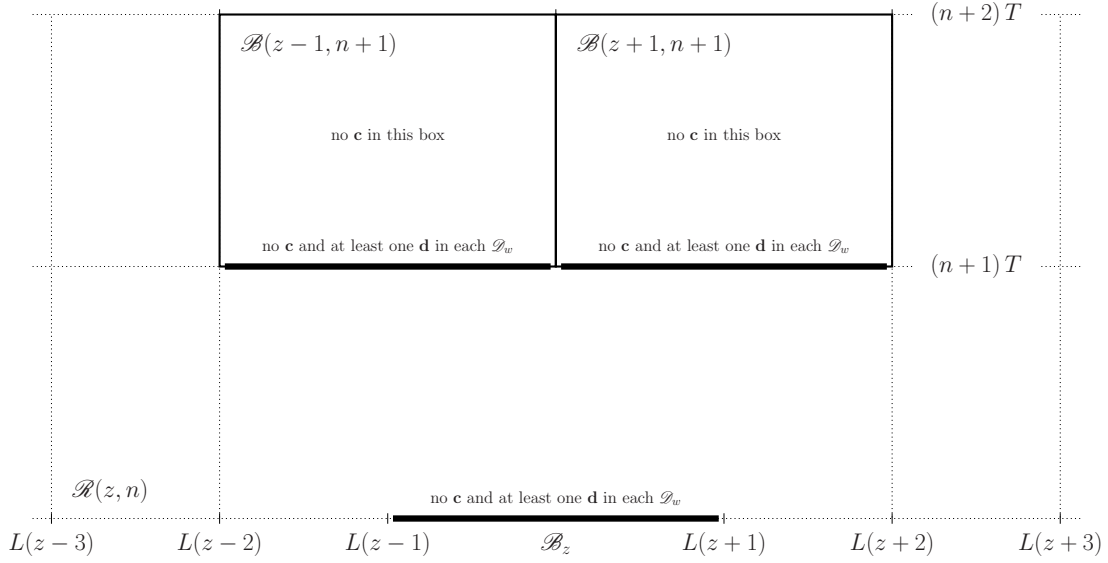}}
\caption{\upshape Illustration of Proposition~\ref{prop:block}.}
\label{fig:block}
\end{figure}


\noindent {\bf Proof of the lower bound}.
 The final step of the proof for the lower bound is similar to the final step of the proof for the upper bound, but using Proposition~\ref{prop:block} together
 with a perturbation argument instead of the three conditions~\eqref{eq:perco-1}--\eqref{eq:perco-3}. let
 $$ \rho = \rc - 2d \rd / (2d - 1). $$
 The goal is to prove that the defectors win for all~$\rho > 0$ sufficiently small.
 From now on, we think of the process with~$\rho > 0$ as being generated by the graphical representation introduced right after~\eqref{eq:equal} supplemented
 with additional dot-arrows.
 More precisely, we let
 $$ C_+ (x, y, z) = \hbox{Poisson process with intensity~$\rho / 4d^2$} \quad \hbox{whenever} \quad x \sim y \ \hbox{and} \ y \sim z. $$
 At the times of the Poisson process~$C_+ (x, y, z)$, we put a dot~$\bullet$ at site~$z$ and draw an arrow~$y \to x$ labelled with a~$\c_+$ to indicate that if
 both site~$y$ and site~$z$ are occupied by cooperators and site~$x$ is empty, then site~$x$ becomes occupied by a cooperator.
 For all~$(z, n) \in \Lat$, we let
 $$ \A_+ (z, n) = \A (z, n) \cap \C (z, n) $$
 where~$\A (z, n)$ is the event in Proposition~\ref{prop:block} and where
 $$ \begin{array}{rcl}
    \C (z, n) & \n : \n & \hbox{there is no~$\c_+$ arrow pointing at the region~$\R (z, n)$}. \end{array} $$
\begin{lemma} --
\label{lem:perturbation}
 Assume~$\b > \b_*$ and~$\rc > 2d \rd / (2d - 1) > 0$, and let~$\ep > 0$. Then,
 $$ P (\A_+ (z, n)) \geq 1 - \ep \quad \hbox{for some~$L$ large and~$\rho > 0$ small}. $$
\end{lemma}
\begin{proof}
 Because Proposition~\ref{prop:block} above holds for all~$\ep > 0$ arbitrarily small, there exists a scaling parameter~$L$ large, fixed from now on, such that
\begin{equation}
\label{eq:perturbation-1}
  P (\A (z, n)) \geq 1 - \ep / 2.
\end{equation}
 The parameter~$L$ being fixed, in view of the common rate of the Poisson processes~$C_+ (x, y, z)$ and the common size of the space-time regions~$\R (z, n)$, we have
\begin{equation}
\label{eq:perturbation-2}
 \begin{array}{l}
   P (\C (z, n)) = P (\poisson (2T (6L + 1)^d \rho) = 0) \vspace*{4pt} \\ \hspace*{80pt}
                 = \exp (- 2L^2 (6L + 1)^d \rho) \geq 1 - \ep / 2 \end{array}
\end{equation}
 for all~$\rho > 0$ small.
 Combining~\eqref{eq:perturbation-1} and~\eqref{eq:perturbation-2}, we deduce that
 $$ \begin{array}{rcl}
      P (\A_+ (z, n)) & \n = \n & P (\A (z, n) \cap \C (z, n)) \vspace*{4pt} \\
                      & \n \geq \n & 1 - (1 - P (\A (z, n))) - (1 - P (\C (z, n))) \geq 1 - \ep/2 - \ep/2 = 1 - \ep. \end{array} $$
 This completes the proof.
\end{proof} \\ \\
 To deduce the lower bound, we say that
 $$ (z, n) \in \Lat \ \hbox{is~$\d$-occupied} \quad \hbox{if and only if} \quad \C_- (z, n) \cap \C_+ (z, n) \ \hbox{occurs} $$
 and let~$X_n = \{z : (z, n) \ \hbox{is~$\d$-occupied} \}. $
 Now, observe that
\begin{equation}
\label{eq:lower-1}
  \begin{array}{rl}
    \A_+ (z, n) = \A (z, n) \cap \C (z, n) & \hbox{is measurable with respect to the graphical} \vspace*{2pt} \\
                                           & \hbox{representation restricted to~$\R (z, n)$} \end{array}
\end{equation}
 because both~$\A (z, n)$ and~$\C (z, n)$ are.
 In addition, by Proposition~\ref{prop:block},
\begin{equation}
\label{eq:lower-2}
  \begin{array}{rcl}
    \A_+ (z, n) \cap \C_+ (z, n) & \n = \n & (\A (z, n) \cap \C (z, n)) \cap \C_+ (z, n) \vspace*{4pt} \\
                                 & \n \subset \n & \C_+ (z \pm e_i, n + 1) \cap \C_- (z \pm e_i, n + 1) \end{array}
\end{equation}
 for all~$i = 1, 2, \ldots, d$, where the events~$\C_- (z, n)$ and~$\C_+ (z, n)$ are for the process with~$\rho > 0$.
 Also, it follows from Lemma~\ref{lem:perturbation} that~$L$ can be chosen such that
\begin{equation}
\label{eq:lower-3}
  P (\A_+ (z, n)) \geq 1 - \ep \quad \hbox{for all~$(z, n) \in \Lat$ and all~$\rho > 0$ small}.
\end{equation}
 As in the proof of Lemma~\ref{lem:perco}, conditions~\eqref{eq:lower-1}--\eqref{eq:lower-3} and~\cite[Theorem~4.3]{durrett_1995} imply that the
 set of~$\d$-occupied sites dominates the set of wet sites in oriented site percolation:
 For all~$\ep > 0$, there exists~$L$ such that, for all~$\rho > 0$ small,
 $$ P (z \in W_n^{\ep}) \leq P (z \in X_n) \ \hbox{for all} \ (z, n) \in \Lat \quad \hbox{whenever} \quad W_0^{\ep} \subset X_0 $$
 where~$W_n^{\ep}$ is now the set of wet sites at level~$n$ in a three-dependent oriented site percolation process where sites are closed with probability~$\ep$.
 Following the exact same approach as in the proofs of Lemmas~\ref{lem:invade} and~\ref{lem:dry}, we conclude that the defectors win whenever
 $$ \b > \b_* \quad \hbox{and} \quad \rd > 0 \quad \hbox{and} \quad \rho = \rc - 2d \rd / (2d - 1) > 0 \quad \hbox{is sufficiently small} $$
 showing that~$\rc^- (\b, \rd) > 2d \rd / (2d - 1)$.
 This completes the proof of Theorem~\ref{th:critical}.
 

\end{document}